\documentclass[11pt]{article}

\usepackage[margin=1in]{geometry}
\usepackage{amsmath,amssymb,amsthm,mathtools,bm}
\usepackage{mathrsfs}
\usepackage{enumitem}
\usepackage{comment}
\usepackage{hyperref}
\usepackage[nameinlink,noabbrev]{cleveref}
\usepackage{tikz}
\usetikzlibrary{arrows.meta}
\usepackage{xcolor}
\usepackage{graphicx}
\usepackage{subcaption}
\usetikzlibrary{decorations.pathmorphing}
\tikzstyle{vertexB}=[circle,draw, minimum size=18pt, scale=0.55, inner sep=0.0pt]
\tikzstyle{vertexS}=[circle,draw, minimum size=18pt, scale=0.52, inner sep=0.0pt]
\tikzstyle{arc}=[->,line width=0.03cm]
\usepackage{thmtools, thm-restate}
\usepackage{float}

\hypersetup{
  colorlinks=true,
  linkcolor=blue!60!black,
  citecolor=blue!60!black,
  urlcolor=blue!60!black
}

\newtheorem{theorem}{Theorem}[section]
\newtheorem{lemma}[theorem]{Lemma}
\newtheorem{proposition}[theorem]{Proposition}

\newtheorem{claim}[theorem]{Claim}
\theoremstyle{definition}

\newenvironment{poc}
  {\begin{proof}[Proof of claim]}
  {\end{proof}}

\newcommand{\disc}{\operatorname{disc}}

\title{A discrepancy dichotomy for 1-factorizations of signed complete bipartite graphs}
\author{Yisai Xue\thanks{School of Mathematics and Statistics, Ningbo University, Ningbo, China. Supported by the National Natural Science Foundation of China (No. 12501486). Email: \texttt{xueyisai@nbu.edu.cn}.}
\hspace{10mm} 
Yacong Zhou\thanks{Shenzhen Institutes of Advanced Technology, Chinese Academy of Sciences. Supported by the National Natural Science Foundation of China (No. 12601694). Email: {\tt yacong.zhou96@gmail.com}}}
\date{\today}

\begin{document}

\maketitle
\begin{abstract}
Given a signing $\sigma\colon E(K_{n,n})\to\{-1,+1\}$ of the complete
bipartite graph, when does $K_{n,n}$ admit a $1$-factorization in
which \emph{every} perfect matching has discrepancy bounded below by a
positive absolute constant?

Unlike the complete-graph case resolved by Ai, He, Im, and
Lee, the bipartite setting carries an unavoidable
obstruction: any \emph{balanced one-sided signing}---one whose edge
signs depend on a single bipartition class, with the two labels split
as evenly as possible---forces every perfect matching to have
discrepancy at most $1/n$. We prove that this is essentially the only obstruction:

For every $0<\varepsilon\leq1$ and all sufficiently large $n$, every signing of $K_{n,n}$ either
\begin{enumerate}
    \item[\rm(i)] admits a $1$-factorization in which every perfect
    matching has discrepancy at least $\varepsilon/53$, or
    \item[\rm(ii)] is $\varepsilon$-close, in normalized Hamming
    distance, to a balanced one-sided signing.
\end{enumerate}
The linear dependence on $\varepsilon$ is optimal up to an absolute factor.

A key ingredient is a rank-one stability argument: if the density of rank-two \(2\times2\) switching patterns is small, then the sign matrix is close in Hamming distance to a rank-one sign matrix.
\end{abstract}
 
\section{Introduction}

The study of discrepancy theory originated with Weyl's 1916
paper on the equidistribution of sequences modulo one~\cite{weyl1916},
and has since developed into a rich subject with applications across
number theory, combinatorics, ergodic theory, discrete geometry, and
statistics; for thorough surveys we refer the reader to the monograph
of Beck and Chen~\cite{beck1987irregularities} and the handbook chapter
of Alexander, Beck, and Chen~\cite{alexander2017geometric}.

Discrepancy questions in graph theory ask whether, given an edge signing
of a large host graph, one can locate a structured subgraph whose signs
exhibit a nontrivial imbalance. More precisely, for a graph $H$ and a
signing $\sigma:E(H)\to\{+1,-1\}$, let
\[
S_\sigma(H):=\sum_{e\in E(H)}\sigma(e).
\]
The \emph{(normalized) discrepancy} of $H$ (with respect to $\sigma$) is then
\[
\disc(H):=\frac{|S_\sigma(H)|}{|E(H)|}.
\]

Many results in this direction show that, under suitable density or degree
conditions, every signing of a host graph contains a single spanning subgraph
of large discrepancy. Such results are known for Hamilton cycles~\cite{freschi2021note,gishboliner2022discrepancies}, perfect
matchings~\cite{balogh2020discrepancies}, spanning trees~\cite{erdos1995discrepancy,erdos1971imbalances}, clique factors~\cite{balogh2021discrepancy}, $H$-factors~\cite{bradac2024hfactor}, powers of Hamilton cycles~\cite{bradac2022powers}, oriented settings~\cite{ai2025oriented,chang2026oretype,FRESCHI2024338,gishboliner2023oriented}, and hypergraph settings~\cite{gishboliner2024tight}.

A common feature of all of these results is that they seek a \emph{single}
subgraph of large discrepancy inside a signed host graph. A more demanding
variant---and the one we pursue here---asks instead for an entire
\emph{decomposition} of $G$ in which \emph{every} piece carries large
discrepancy. A particularly natural setting is that of $1$-factorizations:
given a graph whose edges are signed by $\{-1,+1\}$, can one decompose the
graph into perfect matchings in such a way that \emph{every} matching has
large discrepancy?

For complete graphs, this question was recently studied by Ai, He, Im, and
Lee~\cite{ai2025high}, who proved that, for $n$ sufficiently large, every
signing of $K_{2n}$ admits a $1$-factorization in which every perfect
matching has discrepancy bounded away from zero by a universal constant.
It is natural to ask whether an analogous statement holds for complete
bipartite graphs, the most direct bipartite analogue of $K_{2n}$.

The bipartite setting, however, comes with a natural obstruction: if
$\sigma$ depends only on one side of the bipartition and assigns the two signs in numbers differing by at most one, 
then every perfect matching has discrepancy at most $1/n$. We call such a $\sigma$ a \emph{balanced one-sided signing}
(see Section~\ref{sec:setup} for the formal definition). Our main result
shows that this is essentially the only obstruction.

Let $K_{n,n}=(X\cup Y,E)$ denote the complete bipartite graph with parts
$X$ and $Y$ of size $n$, and let
\[
\mathcal{K}_{n,n}^{\pm}
:=\bigl\{(K_{n,n},\sigma):\sigma\colon E\to\{+1,-1\}\bigr\}
\]
denote the family of \emph{signed complete bipartite graphs}; we write a
generic element as $G=(K_{n,n},\sigma)$. We are now in a position to state our main result.

\begin{theorem}\label{thm-main}
For every $0<\varepsilon\le 1$, there exists
$c=c(\varepsilon)\geq \varepsilon/53$ such that, for all
sufficiently large $n$, every $G=(K_{n,n},\sigma)\in\mathcal{K}_{n,n}^{\pm}$
satisfies at least one of the following:
\begin{enumerate}
    \item[\rm(i)] $G$ admits a $1$-factorization in which every perfect
    matching has discrepancy at least $c$;
    \item[\rm(ii)] $\sigma$ is $\varepsilon$-close, in normalized Hamming
    distance, to a balanced one-sided signing.
\end{enumerate}
\end{theorem}

For sign matrices $A,B\in\{-1,+1\}^{n\times n}$, write
$$
d_H(A,B):=\frac{1}{n^2}
\bigl|\{(i,j):A_{ij}\neq B_{ij}\}\bigr|.
$$
Let $\mathcal B$ be the family of matrices of balanced one-sided
signings, and let $d_H(A,\mathcal B):=\min_{B\in\mathcal B}d_H(A,B)$.

Note that our lower bound on $c(\varepsilon)$ is tight up to a constant. For a $\{\pm1\}$-vector $\mathbf{v}$, let $v^+$ and
$v^-$ denote the proportions of $+1$ and $-1$ entries in $\mathbf{v}$,
respectively. For any rational $0<r<1/2$,  consider a rank-one signing
$M=\mathbf{x}\mathbf{y}^{\top}$ with
\[
x^+=1-\frac{r}{2},\qquad x^-=\frac{r}{2},\qquad
y^+=\frac{1+r}{2},\qquad y^-=\frac{1-r}{2},
\]
where, for simplicity, we choose $n$ such that all the corresponding numbers
of entries are integers.

We first determine the distance from $M$ to a balanced one-sided signing.
Partition the coordinates on which $y_j=+1$ into two sets $Y_1^+$ and
$Y_2^+$ of sizes $n/2$ and $rn/2$, respectively, and let $Y^-$ denote the
set of coordinates on which $y_j=-1$, which has size $(1-r)n/2$. Let
$\mathbf{z}$ be the vector which is $+1$ on $Y_1^+$ and $-1$ on
$Y_2^+\cup Y^-$. Then $\mathbf{z}$ is balanced, since it has exactly $n/2$
entries of each sign. The entries on which $M$ and
$\mathbf{1}\mathbf{z}^{\top}$ differ are illustrated by the shaded regions
in Figure~\ref{fig:counterexample}.

\begin{figure}[ht]
    \centering
    \resizebox{0.88\textwidth}{!}{
    \begin{tikzpicture}[
        >=Latex,
        every node/.style={font=\small}
    ]



    \node at (4,6.1) {$M=\mathbf{x}\mathbf{y}^{\top}$};

    \fill[gray!30] (4,0.5) rectangle (4.8,5);
    \fill[gray!30] (0,0) rectangle (4,0.5);
    \fill[gray!30] (4.8,0) rectangle (8,0.5);

    \draw[thick] (0,0) rectangle (8,5);
    \draw (4,0) -- (4,5);
    \draw (4.8,0) -- (4.8,5);
    \draw (0,0.5) -- (8,0.5);

    \node at (2,2.75) {$+1$};
    \node at (4.4,2.75) {$+1$};
    \node at (6.4,2.75) {$-1$};

    \node at (2,0.25) {$-1$};
    \node at (4.4,0.25) {$-1$};
    \node at (6.4,0.25) {$+1$};

    \node[align=center] at (2,5.45)
        {$Y_1^+$\\[-1mm] $\frac{n}{2}$};
    \node[align=center] at (4.4,5.45)
        {$Y_2^+$\\[-1mm] $\frac{rn}{2}$};
    \node[align=center] at (6.4,5.45)
        {$Y^-$\\[-1mm] $\frac{(1-r)n}{2}$};

    \node[anchor=east,align=right] at (-0.2,2.75)
        {$X^+$\\[-1mm] $\left(1-\frac r2\right)n$};
    \node[anchor=east,align=right] at (-0.2,0.25)
        {$X^-$\\[-1mm] $\frac{rn}{2}$};

    \draw[->,thick] (8.6,2.5) -- (10.4,2.5);
    \node[align=center] at (9.5,3.05)
        {change the\\shaded entries};


    \node at (15,6.1)
        {$\mathbf{1}\mathbf{z}^{\top}$};

    \draw[thick] (11,0) rectangle (19,5);
    \draw (15,0) -- (15,5);
    \draw[dashed] (15.8,0) -- (15.8,5);

    \draw[dashed] (11,0.5) -- (19,0.5);

    \node at (13,2.75) {$+1$};
    \node at (15.4,2.75) {$-1$};
    \node at (17.4,2.75) {$-1$};

    \node at (13,0.25) {$+1$};
    \node at (15.4,0.25) {$-1$};
    \node at (17.4,0.25) {$-1$};

    \node[align=center] at (13,5.45)
        {$Y_1^+$\\[-1mm] $\frac{n}{2}$};
    \node[align=center] at (15.4,5.45)
        {$Y_2^+$\\[-1mm] $\frac{rn}{2}$};
    \node[align=center] at (17.4,5.45)
        {$Y^-$\\[-1mm] $\frac{(1-r)n}{2}$};


    \end{tikzpicture}}

    \caption{
    The rank-one signing $M=\mathbf{x}\mathbf{y}^{\top}$ and a closest
    balanced one-sided signing $\mathbf{1}\mathbf{z}^{\top}$.
    }
    \label{fig:counterexample}
\end{figure}

Indeed, the three shaded regions have normalized sizes $\left(1-\frac r2\right)\frac r2$,
$\frac r2\cdot\frac12$ and $
\frac r2\cdot\frac{1-r}{2}$
respectively. Thus,
\[
d_H(M,\mathbf{1}\mathbf{z}^{\top})
=
\left(1-\frac r2\right)\frac r2
+\frac r4
+\frac{r(1-r)}4
=
r-\frac{r^2}{2}.
\]
Moreover, this is the minimum possible distance to a balanced one-sided
signing of the form $\mathbf{1}\mathbf{z}^{\top}$. Indeed, since
$\mathbf{y}$ has $(1+r)n/2$ positive entries, any balanced $\mathbf{z}$ must
differ from $\mathbf{y}$ in at least $rn/2$ coordinates, and the above choice
achieves equality. A similar calculation shows that the minimum distance to
a balanced one-sided signing of the form $\mathbf{z}\mathbf{1}^{\top}$ is
$(1-r^2)/2$, which is larger than $r-r^2/2$ since $r<1/2$. Hence
\[
d_H(M,\mathcal{B})=r-r^2/2.
\]

On the other hand, every perfect matching has nonnegative signed sum. Indeed,
we have $\sum_j y_j=rn$, while $\mathbf{x}$ has exactly $rn/2$ negative
entries. If a perfect matching $P$ corresponds to a permutation $\pi$, then
\[
S_\sigma(P)
=\sum_i x_i y_{\pi(i)}
=rn-2\sum_{i:x_i=-1}y_{\pi(i)}
\geq rn-2\cdot\frac{rn}{2}=0.
\]
Consequently, for every $1$-factorization
$E(G)=M_1\cup\cdots\cup M_n$, we have
\[
\frac{1}{n}\sum_{i=1}^n\disc(M_i)
=
\frac{S_\sigma(G)}{n^2}
=
r(1-r),
\]
and therefore some perfect matching in the factorization has discrepancy at
most $r(1-r)$. 
Fix $0<\varepsilon<3/8$.
Taking rational values of $r$ arbitrarily close from above to
the solution of $r-r^2/2=\varepsilon$, namely
$r=1-\sqrt{1-2\varepsilon}$,
and passing to the limit gives
\[
c(\varepsilon)
\leq
\lim_{\substack{
r\to 1-\sqrt{1-2\varepsilon}\\
r\in\mathbb{Q}
}} r(1-r)
=2\varepsilon-1+\sqrt{1-2\varepsilon}.
\]
Using the Taylor expansion $
\sqrt{1-2\varepsilon}
=1-\varepsilon-\frac{\varepsilon^2}{2}
-\frac{\varepsilon^3}{2}+O(\varepsilon^4)
$, we obtain
\[
c(\varepsilon)
\leq
\varepsilon-\frac{\varepsilon^2}{2}
-\frac{\varepsilon^3}{2}
+O(\varepsilon^4)
=(1+o(1))\varepsilon.
\]

Our proof proceeds via a structural dichotomy driven by two global
parameters of $\sigma$: its total discrepancy and the density of a certain
local switching pattern. When either parameter is large, we explicitly
construct a $1$-factorization in which every perfect matching has
discrepancy at least $c$, placing $G$ in case~(i). When both parameters
are small, we show that $\sigma$ must already lie close to a balanced
one-sided signing, placing it in case~(ii).

\section{Setup and reduction to three cases}\label{sec:setup}

In this section we fix notation, define the central concepts used
throughout the paper, and reduce Theorem~\ref{thm-main} to three
auxiliary results, proved in Sections~\ref{sec:disclarge},
\ref{sec:nofswitcherslarge}, and~\ref{sec:bothsmall}.

In the proofs, it is convenient to identify a signing $\sigma$
with its \emph{bipartite sign matrix}
$M_\sigma\in\{-1,+1\}^{n\times n}$, whose $(i,j)$ entry is
$\sigma(x_iy_j)$. We write $\mathbf{1}$ for the all-ones vector of
length~$n$. A vector $\mathbf{z}\in\{-1,+1\}^n$ is \emph{balanced} if the numbers of its $+1$ and $-1$ entries differ by at most one. A signing $\sigma$ is a \emph{balanced one-sided signing} if
$M_\sigma\in\{\mathbf{z}\mathbf{1}^{T},\ \mathbf{1}\mathbf{z}^{T}\}$
for some balanced $\mathbf{z}$.

For two signings $\sigma,\sigma'$ we say that $\sigma$ is
\emph{$\varepsilon$-close} to $\sigma'$ if $M_\sigma$ and $M_{\sigma'}$
differ in at most $\varepsilon n^{2}$ entries; $\sigma$ is
$\varepsilon$-close to a family of signings if it is
$\varepsilon$-close to some member of that family.

Figure~\ref{fig:bsa} illustrates a balanced one-sided signing depending
on the $X$-side.

\begin{figure}[H]
\centering
\begin{tikzpicture}[
    scale=0.8,
    x=1cm,y=1cm,
    vertex/.style={circle, fill=black, inner sep=1.6pt},
    block/.style={rounded corners=8pt, draw=black, line width=0.8pt},
    posedge/.style={blue, thick},
    negedge/.style={red, thick}
]

\def\xL{0}
\def\xR{4.4}

\fill[blue!8, rounded corners=8pt] (-0.5,0.2) rectangle (0.5,2.5);
\fill[red!8, rounded corners=8pt] (-0.5,-2.5) rectangle (0.5,-0.2);
\draw[block] (-0.5,-2.5) rectangle (0.5,2.5);
\draw[line width=0.6pt] (-0.5,0) -- (0.5,0);

\fill[gray!5, rounded corners=8pt] (\xR-0.5,-2.5) rectangle (\xR+0.5,2.5);
\draw[block] (\xR-0.5,-2.5) rectangle (\xR+0.5,2.5);

\node[vertex] (xp1) at (\xL,1.8) {};
\node[vertex] (xp2) at (\xL,0.9) {};

\node[vertex] (xm1) at (\xL,-0.9) {};
\node[vertex] (xm2) at (\xL,-1.8) {};

\node[vertex] (y1) at (\xR,1.8) {};
\node[vertex] (y2) at (\xR,0.6) {};
\node[vertex] (y3) at (\xR,-0.6) {};
\node[vertex] (y4) at (\xR,-1.8) {};

\foreach \u in {xp1,xp2}{
    \foreach \v in {y1,y2,y3,y4}{
        \draw[posedge] (\u) -- (\v);
    }
}

\foreach \u in {xm1,xm2}{
    \foreach \v in {y1,y2,y3,y4}{
        \draw[negedge] (\u) -- (\v);
    }
}

\foreach \u in {xp1,xp2,xm1,xm2,y1,y2,y3,y4}{
    \node[vertex] at (\u) {};
}

\node[font=\Large] at (\xL,3.05) {$X$};
\node[font=\Large] at (\xR,3.05) {$Y$};

\node[font=\large, blue!70!black] at (-1.0,1.35) {$X^+$};
\node[font=\large, red!75!black] at (-1.0,-1.35) {$X^-$};

\node[font=\large, blue!70!black] at (2.2,1.6) {$+1$};
\node[font=\large, red!75!black] at (2.2,-1.6) {$-1$};

\end{tikzpicture}
\caption{A balanced one-sided signing of $K_{n,n}$. All edges incident to
vertices in $X^+$ have sign $+1$, while all edges incident to vertices in
$X^-$ have sign $-1$.}
\label{fig:bsa}
\end{figure}

A direct calculation shows that, for any balanced one-sided signing,
every $1$-factorization $\{M_t:t\in[n]\}$ satisfies
\[
\disc(M_t)=
\begin{cases}
0, & n \text{ even},\\
1/n, & n \text{ odd},
\end{cases}
\qquad\text{for every }t\in[n].
\]
In particular, no $1$-factorization of a balanced one-sided signing
can have all of its matchings with discrepancy bounded below by a
positive absolute constant.

We call a 4-cycle $C$ a \emph{switcher} if its 1-factorization $\{\psi_1, \psi_2\}$ satisfies $S_\sigma(\psi_1) \neq S_\sigma(\psi_2)$. If the union $M_1 \cup M_2$ of two matchings contains a switcher with $\psi_i \subseteq M_i$, we may \emph{apply} it by swapping $\psi_1$ and $\psi_2$ between the two matchings; this modifies $M_1$ and $M_2$ simultaneously and thereby adjusts their discrepancy.
We denote by $s(G)$ the number of switchers. 
Our argument tracks two global parameters of the signing: the overall
discrepancy $\disc(G)$ and the number of switchers.

Theorem~\ref{thm-main} reduces to three regimes:
\begin{enumerate}
    \item[\rm(a)] $\disc(G)$ is large;
    \item[\rm(b)] $s(G)$ is large;
    \item[\rm(c)] both $\disc(G)$ and $s(G)$ are small.
\end{enumerate}
In cases~(a) and~(b) we produce a $1$-factorization with uniformly
large discrepancy; in case~(c) we show that the signing must already
be close to a balanced one-sided signing. The corresponding three
auxiliary results follow.

\begin{restatable}{theorem}{theorema}\label{lem:disclarge}
Let $G=(K_{n,n},\sigma)\in\mathcal{K}_{n,n}^{\pm}$ with $n$
sufficiently large. If $\disc(G)=c>0$, then $G$ admits a
$1$-factorization $E(G)=\bigcup_{t=0}^{n-1}E(M_t)$ in which
$|\disc(M_t)-c|\leq 6n^{-1/4}$ for every $t\in\mathbb{Z}_n$.
\end{restatable}


\begin{restatable}{theorem}{theoremb}\label{thm:low-disc-many-sw}
Let $\eta>0$ and $G=(K_{n,n},\sigma)\in\mathcal{K}_{n,n}^{\pm}$ with $n$
sufficiently large. If $s(G)\ge\eta n^{4}$, then $G$ admits a
$1$-factorization $E(G)=\bigcup_{t=1}^{n}E(M_t)$ in which
$\disc(M_t)\ge \eta/2-2/n$ for every $t\in[n]$.
\end{restatable}

\begin{restatable}{theorem}{theoremc}\label{lem:3rdcase}
Let $\alpha>0$, $G=(K_{n,n},\sigma)\in\mathcal{K}_{n,n}^{\pm}$, and $M=M_\sigma$. If
$\disc(G)\le\alpha$ and $s(G)/n^4\le 2\alpha$, then $M$
is $52\alpha$-close to a balanced one-sided signing;
equivalently, $M$ is $52\alpha$-close to either
$\mathbf{1}\mathbf{z}^\top$ or $\mathbf{z}\mathbf{1}^\top$ for some balanced $\pm 1$-vector $\mathbf{z}$.
\end{restatable}

We now deduce Theorem~\ref{thm-main} from these three auxiliary
results.

\begin{proof}[\textbf{Proof of Theorem~\ref{thm-main}}]
Fix $0<\varepsilon\leq 1$, and set $\alpha:=\varepsilon/52$. Choose $n$
sufficiently large so that Theorems~\ref{lem:disclarge},
\ref{thm:low-disc-many-sw}, and \ref{lem:3rdcase} apply, and so that
\[
\max\{6n^{-1/4},2n^{-1}\}\leq \frac{\varepsilon}{52\cdot 53}.
\]
Let $c(\varepsilon):=\frac{\varepsilon}{53}$ and $G=(K_{n,n},\sigma)\in\mathcal{K}_{n,n}^{\pm}$. If
$\disc(G)>\alpha$, then Theorem~\ref{lem:disclarge} gives a
$1$-factorization $E(G)=M_1\cup\cdots\cup M_n$ such that, for every
$i\in[n]$,
\[
\disc(M_i)\geq \disc(G)-6n^{-1/4}
> \frac{\varepsilon}{52}-\frac{\varepsilon}{52\cdot53}
=\frac{\varepsilon}{53}
=c(\varepsilon).
\]
Thus the desired $1$-factorization exists.

If instead $s(G)>2\alpha n^4$, then
Theorem~\ref{thm:low-disc-many-sw}, applied with parameter $2\alpha$,
gives a $1$-factorization $E(G)=M_1\cup\cdots\cup M_n$ such that, for
every $i\in[n]$,
\[
\disc(M_i)\geq \alpha-\frac{2}{n}
\geq \frac{\varepsilon}{52}
-\frac{\varepsilon}{52\cdot53}
=\frac{\varepsilon}{53}
=c(\varepsilon),
\]
so again the desired $1$-factorization exists.

It remains to consider the case that $\disc(G)\leq\alpha$ and
$s(G)\leq2\alpha n^4$. Theorem~\ref{lem:3rdcase} implies that $G$ is
$52\alpha$-close to a balanced one-sided signing. Since
$52\alpha=\varepsilon$, $G$ is $\varepsilon$-close to a balanced
one-sided signing. This completes the proof.
\end{proof}

The remainder of the paper is organized as follows. In Section~\ref{sec:bothsmall} we prove Theorem~\ref{lem:3rdcase}; in
Section~\ref{sec:disclarge} we prove Theorem~\ref{lem:disclarge}; and in
Section~\ref{sec:nofswitcherslarge} we prove
Theorem~\ref{thm:low-disc-many-sw}.

\section{Structural stability and the one-sided obstruction}\label{sec:bothsmall}

For a sign matrix $A$, a switcher in $A$ is said to be of \textit{Type-$i$}
if the corresponding \(2\times2\) sign matrix has rank \(i\) (see Figure~\ref{fig:switchers}). We denote by
\(s_i(A)\) the number of Type-\(i\) switchers in \(A\). For \(G=(K_{n,n},\sigma)\), we write
\(s_i(G):=s_i(M_\sigma)\). Thus \(s(G)=s_1(G)+s_2(G)\).

\begin{figure}[H]
\centering
\begin{subfigure}[b]{0.3\textwidth}
\centering
\begin{tikzpicture}[
    scale=2,
    x=1cm,y=1cm,
    vertex/.style={circle, fill=black, inner sep=1.6pt},
    block/.style={rounded corners=8pt, draw=black, line width=0.8pt},
    posedge/.style={blue!70!black, line width=0.5pt, opacity=0.45},
    negedge/.style={red!75!black, line width=0.5pt, opacity=0.45}
]
  \node[vertex] (v1) at (0,0) {};
  \node[vertex] (v2) at (0,1) {};
  \node[vertex] (v3) at (1,1) {};
  \node[vertex] (v4) at (1,0) {};
  \draw[thick,red] (v1) to (v2);
  \draw[thick,blue] (v2) to (v3);
  \draw[thick,red] (v3) to (v4);
  \draw[thick,blue] (v1) to (v4);
  \node at (0.5,1.25) {$+1$};
  \node at (0.5,-0.25) {$+1$};
  \node at (-0.25,0.5) {$-1$};
  \node at (1.25,0.5) {$-1$};
\end{tikzpicture}
\caption{Type-$1$ switcher.}
\end{subfigure}
\begin{subfigure}[b]{0.45\textwidth}
\centering
\begin{tikzpicture}[
    scale=2,
    x=1cm,y=1cm,
    vertex/.style={circle, fill=black, inner sep=1.6pt},
    block/.style={rounded corners=8pt, draw=black, line width=0.8pt},
    posedge/.style={blue!70!black, line width=0.5pt, opacity=0.45},
    negedge/.style={red!75!black, line width=0.5pt, opacity=0.45}
]
  \node[vertex] (v1) at (0,0) {};
  \node[vertex] (v2) at (0,1) {};
  \node[vertex] (v3) at (1,1) {};
  \node[vertex] (v4) at (1,0) {};
  \node[vertex] (u1) at (2,0) {};
  \node[vertex] (u2) at (2,1) {};
  \node[vertex] (u3) at (3,1) {};
  \node[vertex] (u4) at (3,0) {};
  \draw[thick,blue] (v1) to (v2);
  \draw[thick,blue] (v2) to (v3);
  \draw[thick,red] (v3) to (v4);
  \draw[thick,blue] (v1) to (v4);
  \draw[thick,blue] (u1) to (u2);
  \draw[thick,red] (u2) to (u3);
  \draw[thick,red] (u3) to (u4);
  \draw[thick,red] (u1) to (u4);
  \node at (0.5,1.25) {$+1$};
  \node at (0.5,-0.25) {$+1$};
  \node at (-0.25,0.5) {$+1$};
  \node at (1.25,0.5) {$-1$};
  \node at (2.5,1.25) {$-1$};
  \node at (2.5,-0.25) {$-1$};
  \node at (1.75,0.5) {$+1$};
  \node at (3.25,0.5) {$-1$};
\end{tikzpicture}
\caption{Type-$2$ switchers.}
\end{subfigure}
\caption{The two switcher patterns.}
\label{fig:switchers}
\end{figure}

We briefly outline the strategy of this section. We first use the Type-$2$ switchers to show
directly that $M=M_\sigma$ is close to a rank-one sign matrix
$\mathbf{x}\mathbf{y}^{T}$. We then show that this rank-one approximation still has small
discrepancy and relatively few Type-$1$ switchers. Finally, let $p$ and $q$ be the normalized discrepancy of $\mathbf{x}$ and $\mathbf{y}$, respectively. 
Then, we have
\[
\disc(\mathbf{x}\mathbf{y}^{T})=pq
\qquad\text{and}\qquad
s_1(\mathbf{x}\mathbf{y}^{T})
=\frac{(1-p^2)(1-q^2)}{16}n^4.
\]
These two relations imply that one of $\mathbf{x},\mathbf{y}$ is close to a constant vector, while
the other is close to balanced. Thus $M$ is close to either
$\mathbf{1}\mathbf{z}^{T}$ or $\mathbf{z}\mathbf{1}^{T}$ for some balanced sign vector
$\mathbf{z}$.


\begin{proof}[\textbf{Proof of Theorem~\ref{lem:3rdcase}}]
We first show that there exist two vectors $\mathbf{x}$ and $\mathbf{y}$ in $\{+1,-1\}^n$ such that $d_H(M,\mathbf{x}\mathbf{y}^\top)$ is small.

\begin{claim}\label{cl:goodapprox}
There exist $\mathbf{x}$ and $\mathbf{y}$ in $\{+1,-1\}^n$ such that
$d_H(M,\mathbf{x}\mathbf{y}^\top)\leq 4s_2(G)/n^4$.
\end{claim}

\begin{poc}
Fix a subscript $(r,s)$ of $M$. For every subscript $(i,j)$, one can observe that
$M_{is}M_{rj}M_{rs}\neq M_{ij}$ if and only if the submatrix of $M$ formed by rows $i,r$ and columns $j,s$ corresponds to a Type-2 switcher. Indeed, this is equivalent to
$M_{is}M_{rj}M_{rs}M_{ij}=-1$.

Each Type-2 switcher is counted exactly four times when we vary the choice of $(r,s)$, once for each of its four entries. Therefore,
\[
\sum_{(r,s)}
|\{(i,j):M_{is}M_{rj}M_{rs}\neq M_{ij}\}|
=4s_2(G).
\]
It follows by averaging that there exists a subscript $(r^*,s^*)$ such that
\[
|\{(i,j):M_{is^*}M_{r^*j}M_{r^*s^*}\neq M_{ij}\}|
\leq \frac{4s_2(G)}{n^2}.
\]
Let $B$ be the matrix with
$B_{ij}=M_{is^*}M_{r^*j}M_{r^*s^*}$. Then
$d_H(M,B)\leq 4s_2(G)/n^4$. Moreover, $B=\mathbf{x}\mathbf{y}^\top$ by letting
$x_i=M_{is^*}M_{r^*s^*}$ and $y_j=M_{r^*j}$. This completes the proof.
\end{poc}

Now let $\mathbf{x}$ and $\mathbf{y}$ be the vectors obtained from Claim~\ref{cl:goodapprox}, and let $x^+,x^-$ and $y^+,y^-$ denote the numbers of $+1$ and $-1$ entries in $\mathbf{x}$ and $\mathbf{y}$, respectively. Since $\mathbf{x}\mathbf{y}^\top$ is a rank-one sign matrix, it contains no Type-2 switchers. We next show that $\mathbf{x}\mathbf{y}^\top$ also has small discrepancy and relatively few Type-1 switchers.

\begin{claim}\label{cl:goodapprox2}
$\disc(\mathbf{x}\mathbf{y}^\top)\leq 17\alpha$ and 
$s_1(\mathbf{x}\mathbf{y}^\top)/n^4\leq 5\alpha$.
\end{claim}

\begin{poc}
The first inequality is easy to prove. Changing the sign of one edge changes the normalized discrepancy by at most $2/n^2$. By Claim~\ref{cl:goodapprox},
$d_H(M,\mathbf{x}\mathbf{y}^\top)\leq 4s_2(G)/n^4\leq 4s(G)/n^4\leq 8\alpha$.
It follows that
\[
\disc(\mathbf{x}\mathbf{y}^\top)
\leq \disc(M)+2d_H(M,\mathbf{x}\mathbf{y}^\top)
\leq \alpha+16\alpha=17\alpha.
\]

Let $B=\mathbf{x}\mathbf{y}^\top$. We now prove the second inequality. Consider any $2$ by $2$ submatrix
\[
B'=\begin{bmatrix} B_{rs} & B_{rj}\\ B_{is} & B_{ij}\end{bmatrix}
\]
of $B$ which forms a Type-1 switcher, and let $M'$ be the corresponding $2$ by $2$ submatrix of $M$. The number of such $B'$ for which $M'$ is also a Type-1 switcher is at most $s_1(G)$. It therefore remains to consider those for which $M'$ is either a Type-2 switcher or not a switcher.

Let $a$ be the number of such $B'$ for which $M'$ is a Type-2 switcher, and let $b$ be the number for which $M'$ is not a switcher. If $M'$ is a Type-2 switcher, then the number of entries of $B'$ which differ from the corresponding entries of $M'$ is either $1$ or $3$. In particular, $a\leq s_2(G)$. If $M'$ is not a switcher, then exactly two entries of $B'$ differ from the corresponding entries of $M'$.

On the other hand, every entry on which $B$ and $M$ differ is contained in at most $n^2$ submatrices of order $2$ by $2$. Since $B$ and $M$ differ in $d_H(M,B)n^2$ entries, counting incidences between these differing entries and the above submatrices gives
\[
a+2b\leq d_H(M,B)n^4.
\]
Together with $a\leq s_2(G)$, this implies
\[
a+b\leq \frac{d_H(M,B)n^4+s_2(G)}{2}.
\]
Therefore, by Claim~\ref{cl:goodapprox},
\[
\begin{aligned}
s_1(B)
\leq s_1(G)+\frac{d_H(M,B)n^4+s_2(G)}{2}
\leq s_1(G)+\frac52s_2(G)
\leq \frac52s(G)
\leq 5\alpha n^4,
\end{aligned}
\]
completing the proof.
\end{poc}

Now let $p=\frac{|x^+-x^-|}{n}$ and $q=\frac{|y^+-y^-|}{n}$. By symmetry, we may assume that $p\geq q$. Thus,
\begin{equation}\label{eq:disc(xy)}
    \disc(\mathbf{x}\mathbf{y}^\top)=\frac{|x^+y^++x^-y^--x^+y^--x^-y^+|}{n^2}= \frac{|x^+-x^-||y^+-y^-|}{n^2}=pq. 
\end{equation}

Moreover, a $2$ by $2$ submatrix of $\mathbf{x}\mathbf{y}^\top$ is a Type-1 switcher if and only if its two rows correspond to entries of different signs in $\mathbf{x}$ and its two columns correspond to entries of
different signs in $\mathbf{y}$. Hence
\[
s_1(\mathbf{x}\mathbf{y}^\top)=x^+x^-y^+y^-=
\frac{(1-p^2)(1-q^2)}{16}n^4.
\]
By Claim~\ref{cl:goodapprox2}, it follows that

\begin{equation}\label{eq:pq}
    pq\leq 17\alpha \qquad\text{and}\qquad (1-p^2)(1-q^2)\leq 80\alpha.
\end{equation}

We next show that $B$ is close to a balanced one-sided signing. By replacing both $\mathbf{x}$ and $\mathbf{y}$ by their negatives if necessary, we may assume that $x^+\geq x^-$. Let $\mathbf{z}$ be a balanced vector obtained from $\mathbf{y}$ by changing as few entries as possible and let

\[k=|\{j\in[n]:y_j\neq z_j\}|\]

If $n$ is even, then $\mathbf{y}$ and $\mathbf{z}$ differ in $qn/2$ entries, while if $n$ is odd, they differ in $(qn-1)/2$ entries. Thus, in both cases, we have that $k\leq qn/2$.

For every $i$ with $x_i=1$, the $i$-th row of $\mathbf{x}\mathbf{y}^\top$ differs from the
corresponding row of $\mathbf{1}\mathbf{z}^\top$ in at most $k$ entries, while for every $i$ with $x_i=-1$, they differ in at most $n-k$ entries. Therefore, as $k\leq qn/2$, we have

\begin{equation}\label{eq:dist(xy,1z)}
d_H(\mathbf{x}\mathbf{y}^\top,\mathbf{1}\mathbf{z}^\top)\leq \frac{1}{n^2} \left(k\cdot \frac{(1+p)n}{2} + (n-k) \cdot \frac{(1-p)n}{2} \right)=\frac{1-p}{2}+\frac{kp}{n} \leq \frac{1-p+pq}{2}.
\end{equation}

We now distinguish two cases. First suppose that $p\geq 2/5$. Using $p\geq q$ and \eqref{eq:pq}, we have that
\begin{eqnarray*}
    1-p^2&=&(1-p^2)(1-q^2)+q^2-p^2q^2\\
    &\leq& (1-p^2)(1-q^2)+pq \\
    &\leq& 80\alpha+17\alpha=97\alpha.
\end{eqnarray*}

Since $p\geq 2/5$, we have
$1-p\leq 97\alpha/(1+p)\leq 485\alpha/7$. Hence, by \eqref{eq:pq} and \eqref{eq:dist(xy,1z)},
\[
d_H(\mathbf{x}\mathbf{y}^\top,\mathbf{1}\mathbf{z}^\top)\leq \frac{485}{14}\alpha+\frac{17}{2}\alpha = \frac{302}{7}\alpha.
\]
Together with Claim~\ref{cl:goodapprox}, this gives
\[
d_H(M,\mathbf{1}\mathbf{z}^\top)\leq d_H(M,\mathbf{x}\mathbf{y}^\top)+d_H(\mathbf{x}\mathbf{y}^\top,\mathbf{1}\mathbf{z}^\top)
\leq 8\alpha +\frac{302}{7}\alpha = \frac{358}{7}\alpha
<52\alpha.
\]

It remains to consider the case $p<2/5$. Since $p\geq q$, by \eqref{eq:pq}, we have
\[
(1-p^2)^2\leq (1-p^2)(1-q^2)\leq 80\alpha.
\]
We shall also use the following elementary inequality. Let
$f(p)=\frac{13}{12}(1-p^2)^2-(1-p+p^2)$. For $0\leq p\leq 2/5$, we have
$f''(p)=13p^2-19/3<0$, so $f$ is concave and hence attains its minimum at one of the two endpoints. Since
$f(0)=1/12$ and $f(2/5)=11/2500>0$, we have
\[
1-p+p^2\leq \frac{13}{12}(1-p^2)^2.
\]
Therefore, by \eqref{eq:dist(xy,1z)} and the fact that $p\geq q$, we have
\[
d_H(\mathbf{x}\mathbf{y}^\top,\mathbf{1}\mathbf{z}^\top)
\leq \frac{1-p+p^2}{2}
\leq \frac{13}{24} \cdot 80\alpha = \frac{130}{3}\alpha.
\]
Using Claim~\ref{cl:goodapprox} again, we obtain
\[
d_H(M,\mathbf{1}\mathbf{z}^\top)\leq d_H(M,\mathbf{x}\mathbf{y}^\top)+d_H(\mathbf{x}\mathbf{y}^\top,\mathbf{1}\mathbf{z}^\top)
\leq 8\alpha +\frac{130}{3}\alpha = \frac{154}{3}\alpha<52\alpha.
\]

Thus, in both cases, $M$ is $52\alpha$-close to a balanced one-sided signing. If instead $q>p$, the same argument with $\mathbf{x}$ and $\mathbf{y}$ interchanged shows that $M$ is
$52\alpha$-close to $\mathbf{z}\mathbf{1}^\top$ for some
balanced vector $\mathbf{z}$. This completes the proof.
\end{proof}

\section{High discrepancy (proof of Theorem \ref{lem:disclarge})}\label{sec:disclarge}

We use the following form of Talagrand's inequality, stated
in~\cite{ai2025high} as a simplified version of
McDiarmid's theorem~\cite{mcdiarmid2002concentration}.

\begin{lemma}[Talagrand's inequality on random permutations {\cite{mcdiarmid2002concentration}}]\label{Talagrand}
Let $h : S_n \to \mathbb{R}_{\ge 0}$ be a function, and let $\pi\in S_n$ be chosen uniformly at random. Suppose that there are constants $c, r > 0$ such that the following two conditions hold.
\begin{itemize}
    \item For any permutations $\pi$ and $\pi'$ differing at $d$ coordinates, we have
    $
    |h(\pi) - h(\pi')| \le cd.
    $
    
    \item If $h(\pi) = s$, then we can find a set $S$ with at most $rs$ coordinates such that for any permutation $\pi' \in S_n$ agreeing with $\pi$ on $S$, we have
    $
    h(\pi') \ge s.
    $
\end{itemize}
Let $M$ be the median of $h(\pi)$. Then for every $t > 0$, we have
$\mathbb{P}(|h(\pi) - M| \ge t) \le 6 e^{-\frac{t^2}{16rc^2(M+t)}}.$
\end{lemma}

We will need the following lemma which can be derived from the above Talagrand's inequality. 

\begin{lemma}\label{lem:cyclic-conc}
Let \(A=(a_{ij})_{i,j\in \mathbb Z_m}\) be a \(0\)-\(1\) matrix, and let
\(\pi\in S_m\) be chosen uniformly at random. For \(t\in\mathbb Z_m\), set
\(h_t(\pi):=\sum_{i\in\mathbb Z_m} a_{i,\pi(i+t)} .\)
Then, for \(m\) sufficiently large, with positive probability,
\[
    \bigl|h_t(\pi)-\mathbb E h_t\bigr|\le 3m^{3/4}
    \qquad\text{for every }t\in\mathbb Z_m.
\]
Moreover,
\(\mathbb E h_t=\frac{1}{m}\sum_{i,j\in\mathbb Z_m}a_{ij}\)
for every \(t\).
\end{lemma}

\begin{proof}
For each fixed \(t\in \mathbb Z_m\), let \(m_t\) be a median of \(h_t(\pi)\), where
\(\pi\) is chosen uniformly at random from \(S_m\). We first prove concentration
around \(m_t\).

We verify the hypotheses of Lemma~\ref{Talagrand} with \(c=r=1\).
Suppose that two permutations \(\pi,\pi'\in S_m\) differ on exactly \(d\)
coordinates. Then the terms $a_{i,\pi(i+t)}$
can change only for those \(i\) such that \(\pi(i+t)\ne \pi'(i+t)\). Since the
map \(i\mapsto i+t\) is a bijection of \(\mathbb Z_m\), there are at most \(d\)
such indices \(i\). Hence $|h_t(\pi)-h_t(\pi')|\le d$.

Next suppose that \(h_t(\pi)=s\). Define
\[
    S:=\{\,i+t\in \mathbb Z_m : a_{i,\pi(i+t)}=1\,\}.
\]
Then \(|S|=s\). If another permutation \(\widetilde\pi\in S_m\) agrees with
\(\pi\) on every coordinate in \(S\), then for every \(i\) with
\(i+t\in S\), we still have $a_{i,\widetilde\pi(i+t)}=a_{i,\pi(i+t)}=1$.
Thus \(h_t(\widetilde\pi)\ge s\). This gives the required certificate condition
with \(r=1\).

By Lemma~\ref{Talagrand}, for every \(u>0\), since \(0\le m_t\le m\), we have
\[
    \mathbb P\bigl(|h_t-m_t|\ge u\bigr)
    \le
    6e^{-\frac{u^2}{16(m_t+u)}}\le
    6e^{-\frac{u^2}{16(m+u)}}.
\]
Since $h_t,m_t\in[0,m]$, integrating the tail bound over $[0,m]$
gives
\[
|\mathbb E h_t-m_t|
\leq \mathbb E|h_t-m_t|
\leq 6\int_0^m e^{-u^2/(32m)}\,du
\leq C_0\sqrt m,
\]
for some absolute constant $C_0>0$.

Set \(u=2m^{3/4}\). For all sufficiently large \(m\), we have
\(C_0\sqrt m\le m^{3/4}\). Hence
\[
\begin{aligned}
    \mathbb P\bigl(|h_t-\mathbb Eh_t|\ge 3m^{3/4}\bigr)
    \le
    \mathbb P\bigl(|h_t-m_t|\ge 2m^{3/4}\bigr)  
    \le
    6e^{-\frac{m^{3/2}}{4(m+2m^{3/4})}} 
    \le
    6e^{-\sqrt m/8}.
\end{aligned}
\]
Taking a union bound over all \(t\in\mathbb Z_m\), we obtain
\[
    \mathbb P\left(
    \exists t\in\mathbb Z_m:
    |h_t-\mathbb Eh_t|\ge 3m^{3/4}
    \right)
    \le
    6m e^{-\sqrt m/8}
    =
    o(1).
\]
Thus, for all sufficiently large \(m\), this probability is less than \(1\).
Consequently, there exists a permutation \(\pi\in S_m\) such that for every
\(t\in\mathbb Z_m\),
\[
    |h_t(\pi)-\mathbb Eh_t|
    \le
    3m^{3/4}.
\]

We now compute the expectation of $h_t$. For each fixed \(i\in\mathbb Z_m\), the random
variable \(\pi(i+t)\) is uniformly distributed on \(\mathbb Z_m\). Therefore
\[
    \mathbb E h_t
    =
    \sum_{i\in\mathbb Z_m}
    \mathbb E a_{i,\pi(i+t)}
    =
    \sum_{i\in\mathbb Z_m}
    \frac{1}{m}\sum_{j\in\mathbb Z_m}a_{ij}
    =
    \frac{1}{m}\sum_{i,j\in\mathbb Z_m}a_{ij}.
\]
This completes the proof.
\end{proof}

We restate Theorem \ref{lem:disclarge} for convenience: 

\theorema*

\begin{proof}[Proof of Theorem \ref{lem:disclarge}]
    Let the two vertex classes of $K_{n,n}$ be
    \[
X=\{x_0,x_1,\dots,x_{n-1}\},
\qquad
Y=\{y_0,y_1,\dots,y_{n-1}\}.
\]
Let $p:=|\sigma^{-1}(+1)|/n^2$ be the proportion of +1 entries in the signing $\sigma$; 
replacing $\sigma$ by $-\sigma$ if necessary, we may assume $p \ge 1/2$, so that $c=2p-1\geq 0$.
    For a permutation $\pi\in S_n$ and $t\in \mathbb{Z}_n$, define
$$M_t^{\pi}:=\{x_i y_{\pi(i+t)}: i\in \mathbb{Z}_n\}.$$
Then the family
$\mathcal{M}_{\pi}:=\{M_t^{\pi}: t\in \mathbb{Z}_n\}$
is a $1$-factorization of $K_{n,n}$.

Let \(A=(a_{ij})_{i,j\in\mathbb Z_n}\) be a \(0\)-\(1\) matrix such that
\(
    a_{ij}:=\mathbf 1[\sigma(x_i y_j)=+1]
\)
and $h_t(\pi):=\sum_{i\in\mathbb Z_n} a_{i,\pi(i+t)}$
for each $t\in\mathbb{Z}_n$. Then, \(
    \frac{1}{n}\sum_{i,j\in\mathbb Z_n} a_{ij}=\frac{|\sigma^{-1}(+1)|}{n}=pn
\). In addition, as $h_t(\pi)$ is the number of edges in $M_t^{\pi}$ with sign $+1$, we have

\[\disc(M_t^{\pi})=\frac{|h_t(\pi)-(n-h_t(\pi))|}{n}=\frac{|2h_t(\pi)-n|}{n}.\]

By Lemma~\ref{lem:cyclic-conc}, there exists a permutation \(\pi\in S_n\)
such that
\[
    |h_t(\pi)-pn|\le 3n^{3/4}
    \qquad\text{for every }t\in\mathbb Z_n.
\]
Therefore, for every $t\in\mathbb{Z}_n$, we have that
\[
|\disc(M_t^{\pi})-c|
\leq\frac{|2h_t(\pi)-n-(2p-1)n|}{n}
=\frac{2|h_t(\pi)-pn|}{n}
\leq 6n^{-1/4},
\]
completing the proof.
\end{proof}

\section{Many switchers (proof of Theorem \ref{thm:low-disc-many-sw})}\label{sec:nofswitcherslarge}

In Sections \ref{subsec:ms1} and \ref{subsec:ms2} we state and prove some lemmas needed for the proof of Theorem \ref{thm:low-disc-many-sw} and in Section \ref{subsec:ms4} we give its proof.

\subsection{Decomposing $K_{n,n}-I$ into 2-factors for odd $n$}\label{subsec:ms1}

The following lemma can be deduced from the resolution of the Oberwolfach problem \cite{glock2021resolution}.

\begin{lemma}[Theorem 1.3 of \cite{glock2021resolution}]\label{lem:Kndcom}
    For sufficiently large odd $n$,  $K_n$ has an edge-decomposition
\[
E(K_n)=E(F_1)\cup E(F_2)\cup\cdots\cup E(F_{(n-1)/2}),
\]
where each $F_i\cong\frac{n-3}{4}C_4\cup C_3$ when $n\equiv 3\pmod 4$ and $F_i\cong\frac{n-5}{4}C_4\cup C_{5}$ when $n\equiv 1\pmod 4$. 
\end{lemma}

We will use Lemma~\ref{lem:Kndcom} to obtain a corresponding decomposition
of the crown graph $K_{n,n}-I$, where $I$ is a perfect matching of $K_{n,n}$.
We first recall the following standard construction.

The \emph{canonical double cover} of $G$, denoted by $B(G)$, is the bipartite graph with vertex set
\[
V(B(G))=V\times\{0,1\},
\]
and edge set
\[
E(B(G))
=
\bigl\{
\{(u,0),(v,1)\},\{(u,1),(v,0)\}
:\{u,v\}\in E
\bigr\}.
\]
Equivalently, each edge $\{u,v\}$ of $G$ is replaced by the two edges
$(u,0)(v,1)$ and $(u,1)(v,0)$. 
After labelling the two parts of $K_{n,n}$ by two copies of $V(K_n)$
and taking $I$ to be the diagonal matching, we have
$B(K_n)\cong K_{n,n}-I$.
We will also use the following simple proposition about cycles.

\begin{proposition}\label{prop:cycles}
Let $m \geq 3$.
\begin{enumerate}
    \item[\rm (i)] If $m$ is odd, then $B(C_m) \cong C_{2m}$.
    \item[\rm (ii)] If $m$ is even, then $B(C_m) \cong C_m \cup C_m$.
\end{enumerate}
\end{proposition}

\begin{proof}
    Each step in the canonical double cover changes the second coordinate.
    Following one circuit around $C_m$ therefore returns to the original sheet if and only if $m$ is even.
    If $m$ is even, the two lifts close after $m$ steps and form two vertex-disjoint copies of $C_m$.
    If $m$ is odd, a lift closes only after two circuits and visits all $2m$ vertices, giving a single copy of $C_{2m}$.
\end{proof}




The following lemma shows that after removing a perfect matching, $K_{n,n}$ can be decomposed similarly to $K_n$.
\begin{lemma}\label{lem:crown}
Let $I$ be a perfect matching of $K_{n,n}$. For every sufficiently large odd $n$, $K_{n,n}-I$ admits an edge-decomposition
\[
E(K_{n,n}-I)=E(F_1)\cup E(F_2)\cup\cdots\cup E(F_{(n-1)/2}),
\]
where
\[
F_i \;\cong\;
\begin{cases}
\tfrac{n-3}{2}\,C_4 \cup C_6, & n\equiv 3\pmod 4,\\[2pt]
\tfrac{n-5}{2}\,C_4 \cup C_{10}, & n\equiv 1\pmod 4.
\end{cases}
\]
\end{lemma}

\begin{proof}
By Lemma~\ref{lem:Kndcom}, for $n$ sufficiently large $K_n$ admits a $2$-factorization
\[
E(K_n)=E(F'_1)\cup E(F'_2)\cup\cdots\cup E(F'_{(n-1)/2}),
\]
where
\[
F_i' \;\cong\;
\begin{cases}
\tfrac{n-3}{4}\,C_4 \cup C_3, & n\equiv 3\pmod 4,\\[2pt]
\tfrac{n-5}{4}\,C_4 \cup C_5, & n\equiv 1\pmod 4.
\end{cases}
\]
As $B(K_n)\cong K_{n,n}-I$, under this isomorphism, each $F_i'$ lifts to a spanning $2$-regular subgraph $F_i$ of $K_{n,n}-I$, and the lifts partition $E(K_{n,n}-I)$.

By Proposition~\ref{prop:cycles}, the canonical double cover sends $C_4\mapsto 2C_4$, $C_3\mapsto C_6$, and $C_5\mapsto C_{10}$. Hence, each$\tfrac{n-3}{4}\,C_4\cup C_3$ lifts to $\tfrac{n-3}{2}\,C_4\cup C_6$, and similarly for the other case.
\end{proof}

\subsection{Finding high discrepancy matchings from switchers}\label{subsec:ms2}

We now turn local switchers into global discrepancy by splitting a
switcher-rich $2$-factor into two high-discrepancy perfect matchings.

\begin{lemma}\label{lem:seed}
   Let $\eta>0$ and $G=(K_{n,n},\sigma)\in \mathcal{K}_{n,n}^{\pm}$. If $s(G)\ge \eta n^4$, then there exists a perfect matching $I\subseteq G$ such that $\disc(I)\ge \eta/2$. 
\end{lemma}

\begin{proof}
Choose a maximal family of pairwise vertex-disjoint switchers $Q_1,\ldots,Q_m$.
A fixed $4$-cycle intersects fewer than $2n^3$ different $4$-cycles: indeed, after choosing one of its four vertices, there are at most $(n-1)\binom n2$ choices for a $4$-cycle through that vertex.  By maximality, every switcher intersects some $Q_i$.  Hence
\[
\eta n^4\le s(G)<2mn^3,
\]
so $m\ge \frac{\eta n}{2}$.
For each $i$, write the two perfect matchings of $Q_i$ as $Q_i^+$ and $Q_i^-$ such that
$S_\sigma(Q_i^+)>S_\sigma(Q_i^-)$.
Since $Q_i$ is a switcher,
$S_\sigma(Q_i^+)-S_\sigma(Q_i^-)\in\{2,4\}$.
Delete all vertices covered by $Q_1,\ldots,Q_m$.  The remaining complete bipartite graph has a perfect matching $R$.  Define
\[
M^+:=R\cup Q_1^+\cup\cdots\cup Q_m^+,
\qquad
M^-:=R\cup Q_1^-\cup\cdots\cup Q_m^-.
\]
Then $M^+$ and $M^-$ are perfect matchings and
\[
S_\sigma(M^+)-S_\sigma(M^-)
=\sum_{i=1}^m \big(S_\sigma(Q_i^+)-S_\sigma(Q_i^-)\big)
\ge 2m.
\]
Thus
\[\max\{|S_\sigma(M^+)|,|S_\sigma(M^-)|\}\ge \frac{|\big(S_\sigma(M^+)-S_\sigma(M^-)\big)|}{2}\ge m\ge \frac{\eta n}{2},\]
and therefore one of $M^+,M^-$ is the desired perfect matching, completing the proof.
\end{proof}

\begin{lemma}\label{lem:local-switch}
Let $\alpha>0$, $G=(K_{n,n},\sigma)\in\mathcal K_{n,n}^{\pm}$, and let $F\subseteq G$ be a
spanning $2$-factor of $G$. Suppose that among the connected components of $F$,
at least $\alpha n$ are switcher copies of $C_4$. Then $F$ can be decomposed
into two perfect matchings $M,N$ such that
\[
\min\{\disc(M),\disc(N)\}\geq \frac{\alpha}{2}-\frac{2}{n}.
\]
\end{lemma}

\begin{proof}
Let $F=M_1\cup M_2$ be any $1$-factorization of $F$. For every switcher
$Q$ which is a connected component of $F$, let $\Delta_Q$ be the change in
$S_\sigma(M_1)$ caused by switching on $Q$. The corresponding change in
$S_\sigma(M_2)$ is $-\Delta_Q$, and since $Q$ is a switcher, we have
$\Delta_Q\in\{\pm2,\pm4\}$.

Let $P$ be the sum of all positive $\Delta_Q$, and let $R$ be
the sum of the absolute values of all negative $\Delta_Q$.
Since there are at least $\alpha n$
switchers and $|\Delta_Q|\geq2$ for every one of them, we have
\[
P+R=\sum_Q|\Delta_Q|\geq2\alpha n.
\]

Set $r:=(P+R)/4$ and consider the three shifts $-R$, $\frac{P-R}{2}$, $P$.
Their consecutive gaps are $2r$.
Each of the conditions
$|S_\sigma(M_1)+h|<r$ and $|S_\sigma(M_2)-h|<r$
defines an open interval of length $2r$, which contains at most
one of these three shifts.
Thus some $h_0\in[-R,P]$ satisfies
\[
\min\{|S_\sigma(M_1)+h_0|,\ |S_\sigma(M_2)-h_0|\}
\geq r=\frac{P+R}{4}.
\]

It remains to note that we can obtain a total shift within $2$ of $h_0$ by
applying some of the switchers. If $h_0\geq0$, order the switchers with
positive $\Delta_Q$ arbitrarily and consider their partial sums. These start
at $0$, end at $P$, and increase at each step by either $2$ or $4$, so one of
the partial sums differs from $h_0$ by at most $2$. If $h_0<0$, the same
argument applies to the switchers with negative $\Delta_Q$. Thus, there is a
family of switchers whose total shift $h$ satisfies $|h-h_0|\leq2$.

Apply these switchers and let $M,N$ be the resulting perfect matchings. Then
\[
\min\{|S_\sigma(M)|,|S_\sigma(N)|\}\geq \min\{|S_\sigma(M_1)+h_0|,|S_\sigma(M_2)-h_0|\}-2
\geq\frac{P+R}{4}-2
\geq\frac{\alpha n}{2}-2,
\]
and therefore, $\min\{\disc(M),\disc(N)\}\geq\frac{\alpha}{2}-\frac{2}{n}$, as desired.
\end{proof}

\subsection{Proof of Theorem \ref{thm:low-disc-many-sw}}\label{subsec:ms4}

Write $S_X$ for the set of permutations of a finite set $X$,
and $X^{(k)}$ for the set of ordered $k$-tuples of distinct
elements of $X$.
For integers $n\geq k\geq1$, let
$n^{(k)}:=n(n-1)\cdots(n-k+1)$.
A permutation $\pi\in S_X$ acts coordinatewise on these tuples;
for $\mathcal F\subseteq X^{(k)}$, we write
\[
\pi(\mathcal F):=
\{(\pi(x_1),\ldots,\pi(x_k)):
  (x_1,\ldots,x_k)\in\mathcal F\}.
\]
We interpret $a=b\pm c$ as $|a-b|\leq c$.
We use the following concentration result
from~\cite[Lemma~3.1]{ai2025high}.

\begin{lemma}[\cite{ai2025high}]\label{lem:permutation-concentration}
For all integers $\Delta,k>0$ and every $\xi>0$,
there exist $n_0=n_0(\Delta,\xi,k)>0$ and $c=c(\Delta,\xi,k)>0$ such that the following holds for all $n\ge n_0$ and $p\in[0,1]$. Let $X$ be an $n$-element set and $\mathcal{F}\subseteq X^{(k)}$ be a set of tuples such that $|\mathcal{F}|\ge \xi n$, and let every $x\in X$ be contained in at most $\Delta$ elements of $\mathcal{F}$. Let $f:X^{(k)}\to\{-1,+1\}$ be a signing with $|f^{-1}(+1)|=pn^{(k)}$. Let $\pi\in S_X$ be a permutation chosen uniformly at random. Then with probability at least $1-e^{-cn}$, we have
\[
|f^{-1}(+1)\cap \pi(\mathcal{F})|=(p\pm\xi)|\mathcal{F}|.
\]
\end{lemma}

We are now ready to prove Theorem \ref{thm:low-disc-many-sw}. We restate it for convenience:

\theoremb*

\begin{proof}
If $n$ is even, set $H:=G$ and $t:=n/2$. If $n$ is odd, apply Lemma~\ref{lem:seed} to obtain a perfect matching $I$ of $G$ with $\disc(I)\ge \eta/2$, relabel the vertices so that $I=\{x_i y_i:i\in\mathbb Z_n\}$, and set $H:=G-I$ and $t:=(n-1)/2$. Write the two vertex classes of $H$ as
\[
X=\{x_i:i\in\mathbb Z_n\},\qquad
Y=\{y_i:i\in\mathbb Z_n\}.
\]

The following two claims now hold.

\begin{claim}\label{cl:many-swe}
    If $n$ is even, then $H$ admits a $2$-factorization
    $E(H)=\bigcup_{j=1}^{t} E(F_j)$ with $s(F_j)\ge \eta n$ for every $j\in[t]$.
\end{claim}
\begin{poc}
Let $m := n/2$. Choose $\pi',\pi'' \in S_n$ independently and uniformly at random. For each $i \in \mathbb{Z}_m$, contract $\{x_{\pi'(2i)}, x_{\pi'(2i+1)}\}$ into a single vertex $x_i'$ and $\{y_{\pi''(2i)}, y_{\pi''(2i+1)}\}$ into a single vertex $y_i'$, and set
\[
X'(\pi') = \{x_i' : i \in \mathbb{Z}_m\}, \qquad Y'(\pi'') = \{y_i' : i \in \mathbb{Z}_m\}.
\]
Any fixed pair $\{x_a,x_b\}$ is contracted with probability $\frac{m}{\binom{n}{2}} = \frac{1}{n-1}$, and the same holds on the $Y$-side, independently. Hence for any fixed switcher of $H$,
\[
\mathbb{P}\!\bigl[\text{both its $X$-part and $Y$-part are contracted}\bigr]
\;=\; \frac{1}{(n-1)^2}.
\]
Since $s(H) \ge \eta n^4$, linearity of expectation gives
\[
\mathbb{E}\bigl[\,\#\{\text{switchers with both parts contracted}\}\,\bigr]
\;\ge\; \frac{\eta n^4}{(n-1)^2} \;\ge\; \eta n^2.
\]
Fix $\pi',\pi''$ attaining at least this value, and write $X' = X'(\pi')$, $Y' = Y'(\pi'')$ from now on.

Define a $\{0,1\}$-matrix $A = (a_{ij})_{i,j \in \mathbb{Z}_m}$ by
\[
a_{ij} \;:=\; \mathbf{1}\!\Bigl[\text{the }C_4\text{ in $H$ on }
\{x_{\pi'(2i)}, x_{\pi'(2i+1)}, y_{\pi''(2j)}, y_{\pi''(2j+1)}\}\text{ is a switcher}\Bigr].
\]
Then $\sum_{i,j\in\mathbb Z_m} a_{ij} \ge \eta n^2$, so
\[
    \frac{1}{m}\sum_{i,j\in\mathbb Z_m} a_{ij} \;\ge\; 2\eta n.
\]
For $\pi \in S_m$ and $k\in \mathbb{Z}_m$, we define
\[
    h_k(\pi) \;:=\; \sum_{i\in\mathbb Z_m} a_{i,\pi(i+k)}.
\]
Applying Lemma~\ref{lem:cyclic-conc} to $A$, we obtain a permutation $\pi \in S_m$ such that
\[
    \left|\, h_k(\pi) - \frac{1}{m}\sum_{i,j\in\mathbb Z_m} a_{ij} \,\right|
    \;\le\; 3 m^{3/4}
    \qquad \text{for every } k\in\mathbb Z_m.
\]
In particular, for every $k\in\mathbb Z_m$,
\[
    h_k(\pi)
    \;\ge\;
    \frac{1}{m}\sum_{i,j\in\mathbb Z_m}a_{ij} - 3 m^{3/4}
    \;\ge\;
    2\eta n - 3 m^{3/4}
    \;\ge\; \eta n
\]
for all sufficiently large $n$.

By the definition of $a_{ij}$, for each $k \in \mathbb{Z}_m$ the edge set
\[
E(F_k) \;:=\; \bigcup_{i \in \mathbb{Z}_m} E\!\bigl(C_4 \text{ on } \{x_{\pi'(2i)},\, x_{\pi'(2i+1)},\, y_{\pi''(2\pi(i+k))},\, y_{\pi''(2\pi(i+k)+1)}\}\bigr)
\]
is a $C_4$-factor of $H$ with
\[
s(F_k) \;=\; h_k(\pi) \;\ge\; \eta n,
\]
and $\bigcup_{k \in \mathbb{Z}_m} F_k$ is a $C_4$-factorization of $H$. The family $\{F_k : k \in \mathbb{Z}_m\}$ is therefore the required $2$-factorization. 
\end{poc}

\begin{claim}\label{cl:many-swo}
If $n$ is odd, then $H$ admits a $2$-factorization
$E(H)=\bigcup_{j=1}^{t}E(F_j)$ with $s(F_j)\geq \eta n$
for every $j\in[t]$.
\end{claim}

\begin{poc}
By Lemma~\ref{lem:crown}, $E(H)=\bigcup_{j=1}^{t}E(F_j)$ with
$t=(n-1)/2$, where each $F_j$ is a spanning bipartite $2$-factor of $H$ containing at least $(n-5)/2$ $C_4$-components. In particular, for $n$ sufficiently large, each $F_j$ contains at least $2n/5$ $C_4$-components.

Recall that $H=G-I$ for a perfect matching $I$ of $G$. Every edge of
$G$ lies in at most $n^2$ switchers, and hence at most $n^3$ switchers
of $G$ contain an edge of $I$. Since $s(G)\geq \eta n^4$, for $n$
sufficiently large we have
\[
s(H)\geq s(G)-n^3\geq \frac{3}{4}\eta n^4.
\]

Let
\[
\Sigma_{\mathrm{sw}}
:=
\{(a,b,c,d)\in\mathbb Z_n^{(4)}:
x_ay_bx_cy_dx_a\text{ is a switcher of }H\}.
\]
Each switcher $x_ay_bx_cy_dx_a$ corresponds to exactly four ordered
tuples in $\Sigma_{\mathrm{sw}}$, and therefore
$|\Sigma_{\mathrm{sw}}|=4s(H)\geq 3\eta n^4$. Since
$n^{(4)}\leq n^4$, the density $p:=\frac{|\Sigma_{\mathrm{sw}}|}{n^{(4)}}$
satisfies $p\geq 3\eta$.

For each $j\in[t]$, choose one ordered tuple $(a,b,c,d)$ from each
$C_4$-component of $F_j$ and collect them into a family
$\mathcal C_j^0\subseteq\mathbb Z_n^{(4)}$. Then
$|\mathcal C_j^0|\geq 2n/5$, and, since $F_j$ is a $2$-factor, every
index in $\mathbb Z_n$ appears in at most two tuples of
$\mathcal C_j^0$.

Let $\pi$ be a uniformly random permutation of $\mathbb Z_n$, acting
coordinatewise on the tuples. Define
$f:\mathbb Z_n^{(4)}\rightarrow\{+1,-1\}$ by $f(x)=+1$ if
$x\in\Sigma_{\mathrm{sw}}$ and $f(x)=-1$ otherwise. Applying
Lemma~\ref{lem:permutation-concentration} with $k=4$, $\Delta=2$, and $\xi=\eta/2$, we obtain that, for every fixed $j\in[t]$, with probability at least $1-e^{-cn}$,
\[
|\pi(\mathcal C_j^0)\cap\Sigma_{\mathrm{sw}}|
\geq \left(p-\frac{\eta}{2}\right)|\mathcal C_j^0|
\geq \left(3\eta-\frac{\eta}{2}\right)\frac{2n}{5}
=\eta n.
\]
A union bound over all $j\in[t]$ shows that, with probability
$1-o(1)$, this holds simultaneously for every $j$. Fix such a
permutation $\pi$. Relabelling both vertex classes according to $\pi$
gives a $2$-factorization of $H$ in which every $F_j$ contains at
least $\eta n$ switcher $C_4$-components. Hence $s(F_j)\geq\eta n$
for every $j\in[t]$, as required.
\end{poc}

By Claims~\ref{cl:many-swe} and~\ref{cl:many-swo}, $H$ has a
$2$-factorization into $t$ factors, each containing at least
$\eta n$ switcher components.
Applying Lemma~\ref{lem:local-switch} to each factor yields
$2t$ perfect matchings, each of discrepancy at least
$\eta/2-2/n$.
If $n$ is odd, add the matching $I$, which satisfies
$\disc(I)\geq\eta/2$.
Since $2t=n$ for even $n$ and $2t+1=n$ for odd $n$,
this gives the required $1$-factorization.
\end{proof}

\section*{Declaration on the Use of Generative AI}
The first version of this paper, posted on arXiv, was prepared without AI assistance. In the revised version, we incorporated suggestions from ChatGPT Pro, particularly replacing the original singular-value-based stability proof with a simpler combinatorial averaging argument (Claim 3.1) and improving the dependence of $c(\varepsilon)$ on $\varepsilon$ from quartic to linear. All AI-assisted modifications were carefully verified by the authors, who take full responsibility for the paper.

\bibliographystyle{abbrv}
\bibliography{refs.bib}

\end{document}